\documentclass[a4paper,draft]{amsproc}
\usepackage{amssymb}
\usepackage[hyphens]{url} 

\theoremstyle{plain}
 \newtheorem{thm}{Theorem}[section]
 \newtheorem{prop}{Proposition}[section]
 \newtheorem{lem}{Lemma}[section]
 
\theoremstyle{definition}
 \newtheorem{exm}{Example}[section]
 
\theoremstyle{remark}
 \newtheorem{rem}{Remark}[section]

 \numberwithin{equation}{section}

\renewcommand{\leq}{\leqslant}\renewcommand{\geq}{\geqslant}

\title[A new test for convergence of positive series]{A new test for convergence of positive series}



\author[Abramov]{\bfseries Vyacheslav Abramov} 

\address{ 
24 Sagan Drive \\
Cranbourne North, Victoria-3977 \\
Australia}
\email{vabramov126@gmail.com}

\author[Cadena]{\bfseries Meitner Cadena} 

\address{ 
Departamento de Ciencias Exactas \\
Universidad de las Fuerzas Armadas - ESPE \\
Sangolqui\\
Ecuador
}
\email{mncadena2@espe.edu.ec}

\author[Omey]{\bfseries Edward Omey} 

\address{ 
Faculty of Economics and Business-Campus Brussels \\
KU Leuven at Campus Brussels \\
Brussels\\
Belgium
}
\email{edward.omey@kuleuven.be}

\thanks{Communicated by ...} 

\begin{document}

{\begin{flushleft}\baselineskip9pt\scriptsize
MANUSCRIPT
\end{flushleft}}
\vspace{18mm} \setcounter{page}{1} \thispagestyle{empty}

\begin{abstract}
The paper provides a new test of convergence and divergence of positive series. In particular, it extends the known test by Margaret Martin [
\emph{Bull. Amer. Math. Soc.} \textbf{47}, 452--457 (1941)].
\end{abstract}

\maketitle

\section{Introduction}

The tests for convergence/divergence of positive series have a long history going back to d'Alembert \cite{d'A} and Cauchy \cite{Cauchy}, who established the first most elementary results on their convergence or divergence. The further extensions of the original studies were provided by Raabe, Gauss, Bertrand, De Morgan, Kummer and many other mathematicians. Nowadays there is a large variety of tests on convergence/divergence of positive series, and most of the existing practical problems that involve positive series are resolved. Nevertheless, the problem has a number of important theoretical applications arising in the theory of probability, stochastic processes and their real life applications (e. g. \cite{Abramov,CKO1,Daduna}).

In most of the earlier studies the known tests of convergence/divergence of positive series were supposed to be closely connected with the classes of functions regularly varying at infinity (e.g. Bingham, Goldie and Teugels
 \cite{Bingham}). Recently, Cadena, Kratz and Omey \cite{CKO} described a new class of functions that covers the class of functions regularly varying at infinity, and in the other recent paper of these authors \cite{CKO1} that new class of functions was used for characterization of the tail probability distribution functions under general settings. Taking that new class into consideration enables us to further reconsider and develop the earlier tests on convergence/divergence of positive series. The approach of the present paper is based on studying these problems on convergence/divergence from this new position.

The starting point in the present paper is Raabe's test.\
The test implies a simple logarithmic test, which is known
as Cauchy's second test. This simple test can be extended and leads to a new
test based on logarithms. The same framework has been used by \v{R}eh\'ak \cite{Rehak1,Rehak2} to
extend the formula of Raabe. We show how the new definitions lead to new
convergence/divergence tests. For the undecided cases, we generalize an old
result of Martin \cite{Martin}. In the final remarks, we provide some
one-sided results.

The rest of the paper is organized as follows. In Section \ref{S2}, we first recall Raabe's test, provide its extended version and establish the connection between Raabe's test and a simple $\log$-test. In Section \ref{S3}, we first extend the simple $\log$-test, and on the basis of that extension we derive the main conditions on convergence or divergence of positive series. In Section \ref{S4}, we study the case under which no direct decision can be made. In Section \ref{S5}, we conclude the paper, where the possible development of the theory is discussed, as well as some one-sided results are provided.

\section{A simple log test{}}\label{S2}

The test of Raabe deals with sequences of positive numbers $(a_{n})$. The
sequence is called a Raabe sequence if the following limit exists:%
\begin{equation}
\lim_{n\rightarrow \infty }n\left(\frac{a_{n+1}}{a_{n}}-1\right)=\theta.
\label{1}
\end{equation}%
In traditional applications of Raabe sequences, limit relation \eqref{1} implies that
\begin{equation}\label{1.5}
\begin{cases}\sum_{i=1}^{\infty}a_{i}=\infty, &\text{if} \ \theta >-1,\\
\sum_{i=1}^{\infty }a_{i}<\infty, &\text {if} \ \theta <-1,\\
\text{no decision can be made}, &\text{if} \ \theta=-1.
\end{cases}
\end{equation}
%
%
%
%
For a recent review of Raabe's test, we refer to Hammond \cite{Hammond}.

However actually limit relation \eqref{1} is more informative than that is presented by \eqref{1.5}.
It is well-known
that \eqref{1} implies that $(a_{n})$ is a regularly varying sequence (e.g. Bingham, Goldie and Teugels
 \cite[Chapter 1.9]{Bingham} or Bojanic and Seneta \cite{BS}), and Karamata's theorem (see \cite[Chapter 1.9]{Bingham}) can be used for establishing the properties of partial sums.
Namely, we have the following result.

\begin{lem}\label{lem1}
Assume that \eqref{1} holds.

\begin{enumerate}
\item [$(i)$] If $\theta >-1$, then $\sum_{i=1}^{n}a_{i}\rightarrow \infty $ and $%
\sum_{i=1}^{n}a_{i}\sim na_{n}/(1+\theta )$.

\item [$(ii)$] If $\theta <-1$, then $\sum_{i=1}^{\infty }a_{i}<\infty $ and $%
\sum_{i=n}^{\infty }a_{i}\sim -na_{n}/(1+\theta )$.

\item [$(iii)$] If $\theta =-1$, then test \eqref{1} is inconclusive.
\end{enumerate}
\end{lem}

Lemma \ref{lem1} shows not only convergence/divergence of $\sum_{i=1}^{n}a_{i}$,
but also the precise rate at which this happens.

Now we rewrite (1) by using logarithms. First observe that $n\ln
(1+1/n)\rightarrow 1$. Also observe that \eqref{1} implies that $%
a_{n+1}/a_{n}\rightarrow 1$. Since $\ln (z)\sim z-1$ as $z\rightarrow 1$, then it
follows that \eqref{1} is equivalent to%
\begin{equation}
\lim_{n\rightarrow \infty }\frac{\ln (a_{n+1}/a_{n})}{\ln (w(n+1)/w(n))}%
=\lim_{n\rightarrow \infty }\frac{\bigtriangleup \ln a_{n}}{\bigtriangleup
\ln w(n)}=\theta.  \label{2}
\end{equation}%
where $w(n)=n,n\geq 1$, and $\bigtriangleup \alpha _{n}=\alpha
_{n+1}-\alpha _{n}$.

By using the Stolz-C{e}s\`{a}ro lemma and taking sums in \eqref{2} we
obtain
\begin{equation}
\lim_{n\rightarrow \infty }\frac{\ln a_{n}}{\ln w(n)}=\theta.
\label{3}
\end{equation}

For further use, we denote by $RV_{\alpha }$ the class of regularly varying
functions of index $\alpha $. The integer part of $x$ is denoted by $\left[ x%
\right] $.

Whenever $\lim \sup_{x\rightarrow \infty }f(x)/g(x)<\infty $, we write $%
f(x)\preceq g(x)$. The relation " $\preceq $ " is a partial order. If $%
f(x)\preceq g(x)$ and $g(x)\preceq f(x)$, the functions $f(x)$ and $g(x)$
are called equivalent and we write $f(x)\asymp g(x)$. If $f(x)\preceq g(x)$
and $g(x)\preceq h(x)$, then also $f(x)\preceq h(x)$.

Cadena, Kratz and Omey \cite{CKO} showed that \eqref{3} (with $w(n)=n$) holds if and only
if $f(x)=a_{\left[ x\right] }$ satisfies the following property.

\begin{lem}\label{lem2}
Assume that $w(x)=x$, and let $f(x)=a_{\left[ x\right] }$. Then \eqref{3} holds if
and only if there exist functions $A(x),B(x)\in RV_{\theta }$ so that $%
A(x)\preceq f(x)\preceq B(x)$. Moreover we have:

\begin{enumerate}
\item [$(i)$] If $\theta >-1$, then $\sum_{i=1}^{n}a_{i}\rightarrow \infty $, $%
nA(n)\preceq \sum_{i=1}^{n}a_{i}\preceq nB(n)$, and%
\[
\lim_{n\rightarrow \infty }\frac{\ln (\sum_{i=1}^{n}a_{i})}{\ln n}=\theta +1%
\text{.}
\]

\item [$(ii)$] If $\theta <-1$, then $\sum_{i=1}^{\infty }a_{i}<\infty $, $%
nA(n)\preceq \sum_{i=n}^{\infty }a_{i}\preceq nB(n)$, and%
\[
\lim_{n\rightarrow \infty }\frac{\ln (\sum_{i=n}^{\infty }a_{k})}{\ln n}%
=\theta +1\text{.}
\]
\end{enumerate}
\end{lem}

This test about convergence/divergence of the series $\sum a_{i}$ is
sometimes called Cauchy's second test \cite{CKO}. It was re-invented, for example,
in Rao \cite{Rao}. Here in Lemma \ref{lem2} the asymptotic estimates for the partial sums are added.

\section{An extension}\label{S3}

We reconsider \eqref{3} for a general type of the functions $w(x)$. We make the
following assumptions:
\begin{enumerate}
\item [$(a)$] $w(x)\uparrow \infty $ is strictly increasing; the inverse of $w(x)$ is
denoted by $w^{\mathrm{i}}(x)$.

\item[$(b)$] $\forall y$ we have $\lim_{x\rightarrow \infty }w(x+y)/w(x)=1$.
\end{enumerate}

In the sequel we shall assume that $w(x)$ satisfies these
assumptions. Under the assumption that \eqref{3} holds we have the following new result.

\begin{prop}\label{prop3}
 We take $f(x)=a_{\left[ x\right] }$. The following are equivalent:

\begin{enumerate}
\item [$(i)$]%
\begin{equation}
\lim_{n\rightarrow \infty }\frac{\ln a_{n}}{\ln w(n)}=\theta \text{,}
\label{4}
\end{equation}

\item [$(ii)$] There exist functions $A(x),B(x)\in RV_{\theta }$ such that
\begin{equation}
A(w(n))\preceq a_{n}\preceq B(w(n))\text{.}  \label{5}
\end{equation}
\end{enumerate}
\end{prop}

\begin{proof}
Using $f(x)=a_{\left[ x\right] }$ we see that \eqref{4} holds if and only if \\
$\ln f(x)/\ln
w(x)\rightarrow \theta $. Replacing $x$ by $w^{\mathrm{i}}(x)$ it follows that%
\[
\lim_{x\rightarrow \infty }\frac{\ln f(w^{\mathrm{i}}(x))}{\ln x}=\theta \text{.}
\]%
As in Lemma \ref{lem2}, from Cadena, Kratz and Omey \cite{CKO}, we obtain $A(x)\preceq
f(w^{\mathrm{i}}(x))\preceq B(x)$ with $A,B\in RV_{\theta }$, and \eqref{5} follows.
Starting from \eqref{5} we use a property of regular variation: If $U(x)\in
RV_{\alpha }$, then $\ln U(x)/\ln x\rightarrow \alpha $, see \cite{Bingham}, to obtain
\eqref{4}.
\end{proof}

Now we reconsider \eqref{2} for general $w(n)$. Since the requirement is stronger
than \eqref{4}, we obtain the stronger result. The result has been stated and proved in
\cite{Rehak1}, but we provide an alternative proof that has the advantage that it can be easily extended in order to
obtain one-sided results given in the concluding remarks.

\begin{prop}\label{prop4}
Assume that \eqref{2} holds, and let $f(x)=a_{\left[ x\right] }$. Then $f(x)$
can be presented in the form $f(x)=h(w(x))$, where $h(x)$ is regularly varying with index $%
\theta $, and the following representation holds:%
\[
f(x)=a(x)\exp \int_{a}^{w(x)}\lambda (y)\frac{1}{y}\mathrm{d}y\text{, }x\geq a\text{,}
\]%
in which $a(x)\rightarrow a>0$, and $\lambda (x)\rightarrow \theta $, as $%
x\rightarrow \infty $.
\end{prop}

\begin{proof}
We start from \eqref{2} and write
\[
\ln \left(\frac{a_{n+1}}{a_{n}}\right)=\theta (n)\ln \left(\frac{w(n+1)}{w(n)}\right),
\]%
where $\theta (n)\rightarrow \theta $ as $n\rightarrow \infty $. For $%
\epsilon >0$, we choose $n%
{{}^\circ}%
$ so that $\theta -\epsilon \leq \theta (n)\leq \theta +\epsilon$, $\forall
n\geq n%
{{}^\circ}%
$. Taking sums, we find that for $M>N\geq n%
{{}^\circ}%
$,
\[
(\theta -\epsilon )\sum_{i=N}^{M-1}\ln \left(\frac{w(i+1)}{w(i)}\right)\leq
\sum_{i=N}^{M-1}\ln \left(\frac{a_{i+1}}{a_{i}}\right)\leq (\theta +\epsilon
)\sum_{i=N}^{M-1}\ln \left(\frac{w(i+1)}{w(i)}\right),
\]%
or%
\[
(\theta -\epsilon )\ln \left(\frac{w(M)}{w(N)}\right)\leq \ln \left(\frac{a_{M}}{a_{N}}\right)\leq
(\theta +\epsilon )\ln \left(\frac{w(M)}{w(N)}\right).
\]%
Using $f(x)=a_{\left[ x\right] }$, we find that for $y>x>n%
{{}^\circ}%
$,
\[
(\theta -\epsilon )\ln \left(\frac{w(\left[ y\right] )}{w(\left[ x\right] )}\right)\leq
\ln \left(\frac{f(y)}{f(x)}\right)\leq (\theta +\epsilon )\ln\left( \frac{w(\left[ y\right] )}{%
w(\left[ x\right] )}\right).
\]%
We continue with the inequality on the right hand side of this expression.
It follows that
\[
\ln \left(\frac{f(y)}{f(x)}\right)\leq (\theta +\epsilon )\ln \left(\frac{w(y)}{w(x)}\right)+(\theta
+\epsilon )\ln \left(\frac{w(\left[ y\right] )w(x)}{w(\left[ x\right] )w(y)}\right).
\]%
For $x,y$ sufficiently large, we obtain
\[
\ln \left(\frac{f(y)}{f(x)}\right)\leq \epsilon +(\theta +\epsilon )\ln \left(\frac{w(y)}{w(x)}%
\right),
\]%
or equivalently%
\[
\ln \left(\frac{f(w^{\mathrm{i}}(y))}{f(w^{\mathrm{i}}(x))}\right)\leq \epsilon +(\theta +\epsilon )\ln \left(%
\frac{y}{x}\right).
\]

Now we fix $t>1$ and replace $y$ by $y=xt$. For $x$ sufficiently large, we
find%
\[
\ln \left(\frac{f(w^{\mathrm{i}}(tx))}{f(w^{\mathrm{i}}(x))}\right)\leq \epsilon +(\theta +\epsilon )\ln t%
\text{.}
\]%
In a similar way we also obtain%
\[
-\epsilon +(\theta -\epsilon )\ln t\leq \ln \left(\frac{f(w^{\mathrm{i}}(tx))}{f(w^{\mathrm{i}}(x))}\right)%
\text{.}
\]%
Since $\epsilon $ is arbitrary, we conclude that%
\[
\lim_{x\rightarrow \infty }\ln \left(\frac{f(w^{\mathrm{i}}(tx))}{f(w^{\mathrm{i}}(x))}\right)=\theta \ln
t.
\]%
It follows that $f(w^{\mathrm{i}}(x))$ is regularly varying with index $\theta $. The
representation theorem in \cite{Bingham} finalizes the proof of the result.
\end{proof}

\begin{rem}\label{rem4.5}
 Assume that $a_{n}=f(w(n))$, where $f(x)$ is a normalized regularly
varying function, i.e. $f(x)$ satisfies $xf^{\prime }(x)/f(x)\rightarrow
\theta $ as $x\rightarrow \infty $. In this case we have $\bigtriangleup \ln
a_{n}=\ln f(w(n+1))-\ln f(w(n))$. Since $(\ln f(x))^{\prime }=f^{\prime
}(x)/f(x)$, the mean value theorem yields%
\[
\bigtriangleup \ln a_{n}=\frac{f^{\prime }(\alpha _{n})}{f(\alpha _{n})}%
(w(n+1)-w(n)),
\]%
where $w(n)\leq \alpha _{n}\leq w(n+1)$. Since $w(n+1)\sim w(n)$, we have $%
\alpha _{n}\sim w(n)$ and it follows that%
\[
\bigtriangleup \ln a_{n}=\frac{\alpha _{n}f^{\prime }(\alpha _{n})}{f(\alpha
_{n})}\left(\frac{w(n+1)}{w(n)}-1\right)\frac{w(n)}{\alpha _{n}}.
\]%
Using $\bigtriangleup \ln w(n)\sim (w(n+1)/w(n)-1)$, we conclude that \eqref{2}
holds.
\end{rem}

Now we generalize Lemma \ref{lem2} as follows. The following test is new and to
our knowledge has not been stated yet.

\begin{thm}\label{thm5}
Assume that%
\begin{equation}
\lim_{n\rightarrow \infty }\frac{\ln (a_{n}/\bigtriangleup w(n))}{\ln w(n)}%
=\theta \text{.}  \label{6}
\end{equation}%
Then there exist functions $A(x),B(x)\in RV_{\theta }$ so that the following
holds:
\begin{enumerate}
\item [$(i)$] If $\theta <-1$, then $\sum_{i=1}^{\infty }a_{i}<\infty $,\\ $%
w(n)A(w(n))\preceq \sum_{i=n}^{\infty }a_{i}\preceq w(n)B(w(n))$, and%
\[
\lim_{n\rightarrow \infty }\frac{\ln (\sum_{i=n}^{\infty }a_{i})}{\ln w(n)}%
=\theta +1\text{.}
\]

\item [$(ii)$] If $\theta >-1$, then $\sum_{i=1}^{\infty }a_{i}=\infty $,\\ $%
w(n)A(w(n))\preceq \sum_{i=1}^{n}a_{i}\preceq w(n)B(w(n))$, and%
\[
\lim_{n\rightarrow \infty }\frac{\ln (\sum_{i=1}^{n}a_{i})}{\ln w(n)}=\theta +1%
\text{.}
\]
\end{enumerate}
\end{thm}

\begin{proof}
From Proposition \ref{prop3} and \eqref{6} we have
\[
A(w(n))\preceq \frac{a_{n}}{\bigtriangleup w(n)}\preceq B(w(n))\text{,}
\]%
where $A,B\in RV_{\theta }$. It follows that $\bigtriangleup
w(n)A(w(n))\preceq a_{n}\preceq \bigtriangleup w(n)B(w(n))$. Using the
regular variation of $A$ and $B$ and using $w(n+1)\sim w(n)$, we find that%
\[
\int_{w(n)}^{w(n+1)}A(z)\mathrm{d}z\preceq a_{n}\preceq \int_{w(n)}^{w(n+1)}B(z)\mathrm{d}z%
\text{.}
\]

Now first assume that $\theta <-1$. In this case $\int_{b}^{\infty
}A(z)\mathrm{d}z+\int_{b}^{\infty }B(z)\mathrm{d}z<\infty $, and
\[
\int_{x}^{\infty }A(z)\mathrm{d}z\sim -\frac{xA(x)}{\theta +1}\text{, }%
\int_{x}^{\infty }B(z)\mathrm{d}z\sim -\frac{xB(x)}{\theta +1}\text{.}
\]%
It follows that $\sum_{i=i%
{{}^\circ}%
}^{\infty }a_{i}<\infty $ and%
\[
\int_{w(n)}^{\infty }A(z)\mathrm{d}z\preceq \sum_{i=n}^{\infty }a_{i}\preceq
\int_{w(n)}^{\infty }B(z)\mathrm{d}z\text{,}
\]%
so that%
\[
w(n)A(w(n))\preceq \sum_{i=n}^{\infty }a_{i}\preceq w(n)B(w(n))\text{.}
\]

If $\theta >-1$, we have
\[
\int_{b}^{x}A(z)\mathrm{d}z\sim \frac{xA(x)}{\theta +1}\text{, }\int_{b}^{x}B(z)\mathrm{d}z%
\sim \frac{xB(x)}{\theta +1},
\]%
and now it follows that
\[
w(n)A(w(n))\preceq \sum_{i=1}^{n}a_{i}\preceq w(n)B(w(n))\text{.}
\]%
This proves the result.
\end{proof}

\begin{rem}\label{rem1}
Theorem \ref{5} not only provides conditions for convergence and divergence, but
also provides estimates for the partial sums.
\end{rem}
\begin{rem}
In Bourchtein et al. \cite{BBNV}, the authors consider a function $F(x)>0$ so
that $F^{\prime }(x)>0$ is nonincreasing and $\sum_{i=1}^{\infty }F^{\prime
}(i)=\infty $. Then the authors consider sequences $(a_{n})$ of positive
numbers so that the limit
\[
\lim_{n\rightarrow \infty }\frac{\ln (a_{n}/F^{\prime }(n))}{\ln F(n)}%
=\theta
\]%
exists. The conclusions about convergence or divergence of $\sum a_{i}$ are the
same as in Theorem \ref{thm5}.
\end{rem}

\begin{exm}
 Take $w(n)=\ln n$. We have%
\[
w(n+1)-w(n)=\ln \left(1+\frac{1}{n}\right)=\frac{1}{n}-\frac{1}{2n^{2}}+o\left(\frac{1}{n^2}\right)\text{.%
}
\]%
Assumption \eqref{6} in this case is%
\[
\lim_{n\rightarrow \infty }\frac{\ln (a_{n}/\ln (1+1/n))}{\ln \ln n}=\theta
\text{.}
\]%
We have $\ln (a_{n}/\ln (1+1/n))=\ln (na_{n})-\ln n\ln (1+1/n))$. Now note
that we have the following expansion:
\[
\ln \left(n\ln \left(1+\frac{1}{n}\right)\right)=\ln \left(1-\frac{1}{2n^{2}}+o\left(\frac{1}{n^2}\right)\right)=O\left(\frac{1}{n^2}\right).
\]%
Hence the condition can be simplified and given by%
\[
\lim_{n\rightarrow \infty }\frac{\ln (na_{n})}{\ln \ln n}=\theta \text{.}
\]
\end{exm}

\begin{exm}
We study the sequence $a_{n}=(\ln n)^{\theta }/n$. In this case \eqref{2} leads
to $\ln a_{n}/\ln n\rightarrow -1$, and we can not decide about
convergence or divergence of $\sum a_{n}$. Using the new test, we have $%
\ln (na_{n})=\theta \ln \ln n$, and we have convergence/divergence depending
on $\theta <-1$ resp. $\theta >-1$.
\end{exm}

\begin{exm}
Taking $w(n)=\ln (\ln n)$, we find that \eqref{6} leads to%
\[
\lim_{n\to\infty} \frac{\ln ((n\ln n)a_{n})}{\ln (\ln (\ln n))}=\theta \text{,}
\]%
and we have convergence/divergence of the series when $\theta <-1$, resp. $%
\theta >-1$.
It is not hard to extend this, cf.\ Martin \cite{Martin}.
\end{exm}

Using Proposition \ref{prop4}, we have the following theorem presented below. The main point of the
next theorem is that it not only provides a condition to conclude
convergence or divergence, but also gives information about the rate at
which this happens. This result is also available in \cite{Rehak2}.

\begin{thm}\label{thm6}
Let $b_{n}=a_{n}/\bigtriangleup w(n)$ and assume that
\[
\lim_{n\rightarrow \infty }\frac{\bigtriangleup \ln b_{n}}{\bigtriangleup
\ln w(n)}=\theta \text{.}
\]

\begin{enumerate}
\item [$(i)$] If $\theta >-1$, then $\sum_{i=1}^{\infty }a_{i}=\infty $, and%
\[
\sum_{i=1}^{n}a_{i}\sim \frac{1}{1+\theta }w(n)h(w(n))\sim \frac{1}{1+\theta }%
\cdot\frac{w(n)}{w(n+1)-w(n)}a_{n}\text{.}
\]

\item [$(ii)$] If $\theta <-1$, then $\sum_{i=1}^{\infty }a_{i}<\infty $, and%
\[
\sum_{i=n}^{\infty }a_{i}\sim -\frac{1}{1+\theta }w(n)h(w(n))\sim \frac{-1}{%
1+\theta }\cdot\frac{w(n)}{w(n+1)-w(n)}a_{n}\text{.}
\]
\end{enumerate}
\end{thm}

\begin{proof}
Let $f(x)=b_{\left[ x\right] }$. From Proposition \ref{prop4} we have $%
f(x)=h(w(x))$, where $h(x)\in RV_{\theta }$. Hence,
\[
\frac{a_{n}}{w(n+1)-w(n)}=h(w(n))\text{,}
\]%
so that $a_{n}=(w(n+1)-w(n))h(w(n))$. For $n\rightarrow \infty $, we find
that, as $n\rightarrow \infty $, $a_{n}\sim \int_{w(n)}^{w(n+1)}h(z)\mathrm{d}z$. Now
the result follows from Karamata's theorem.
\end{proof}

\begin{rem}
It is shown in \v{R}eh\'ak \cite{Rehak2} that \eqref{2} is equivalent to Kummer's
test. Compared to Kummer's test, we obtained the explicit
expressions for the partial sums.
\end{rem}

\section{The undecided case $\protect\theta =-1$}\label{S4}

\subsection{Results related to Theorem \ref{thm5}}

Let $\alpha (n)$ be defined as%
\[
\alpha (n)=\frac{\ln (a_{n}/\bigtriangleup w(n))}{\ln w(n)}\text{.}
\]%
If $\alpha (n)\rightarrow \theta =-1$, then Theorem \ref{thm5} doesn't lead to the decision.

We prove three types of results.

a) In the first type of results, we assume that $\alpha (n)+1\rightarrow 0$ at certain rate.
Apparently,%
\[
(\alpha (n)+1)\ln w(n)=\ln \left(\frac{a_{n}}{\bigtriangleup w(n)}\right)+\ln w(n)=\ln
\left(\frac{w(n)a_{n}}{\bigtriangleup w(n)}\right),
\]%
and then
\[
\frac{a_{n}}{\bigtriangleup w(n)w^{\mathrm{i}}(n)}=\exp (\alpha (n)+1)\ln w(n)\text{.%
}
\]

\begin{prop}\label{prop7}
\begin{enumerate}
\item []
\item [$(i)$] Assume that $\exp (\alpha (n)+1)\ln w(n)>B>0$. Then $%
\sum_{i=1}^{\infty }a_{i}=\infty $ and $\sum_{i=a}^{n}a_{i}\succeq \ln w(n)$.

\item [$(ii)$] Assume that $\exp (\alpha (n)+1)\ln w(n)\rightarrow C$ where $%
0<C<\infty $. Then $\sum_{i=1}^{\infty }a_{i}=\infty $ and $%
\sum_{i=1}^{n}a_{i}\sim C\ln w(n)$.
\end{enumerate}
\end{prop}

\begin{proof}
$(i)$\ If $\exp (\alpha (n)+1)\ln w(n)>B>0$ then%
\[
a_{n}\geq B\bigtriangleup w(n)w^{\mathrm{i}}(n)\succeq \int_{w(n)}^{w(n+1)}\frac{1}{z%
}\mathrm{d}z\text{.}
\]%
It follows that $\sum_{i=a}^{N}a_{i}\succeq \int_{b}^{w(N)}z^{-1}\mathrm{d}z$, and hence
$\sum_{i=1}^{N}a_{i}\succeq \ln w(N)$.

$(ii)$ We have%
\[
a_{n}\sim C\bigtriangleup w(n)w^{\mathrm{i}}(n)\sim C\int_{w(n)}^{w(n+1)}\frac{1}{z}dz%
\text{.}
\]%
The result follows by summation.
\end{proof}

\begin{exm}
We study $a_{n}=n^{-1}(\ln n)^{p}$.

Using $w(n)=n$ we have $\bigtriangleup \ln a_{n}/\bigtriangleup \ln
w(n)\rightarrow -1$, that is inconclusive case.

Now we take $w(n)=\ln n$. We find:
\begin{eqnarray*}
\ln (a_{n}/\bigtriangleup w(n)) &=&\ln a_{n}-\ln \bigtriangleup w(n) \\
&=&-\ln n+p\ln \ln n-\ln \left(\ln \left(1+\frac{1}{n}\right)\right) \\
&=&p\ln \ln n-\ln \left(n\ln \left(1+\frac{1}{n}\right)\right) \\
&=&p\ln \ln n-\ln \left(1+\left(n\left(\ln (1+\frac{1}{n}\right)\right)-1\right),
\end{eqnarray*}%
and then%
\[
\ln (a_{n}/\bigtriangleup w(n))-p\ln \ln n\sim n\ln \left(1+\frac{1}{n}\right)-1\sim -%
\frac{1}{2n}.
\]%
We find
\[
\alpha (n)=\frac{\ln (a_{n}/\bigtriangleup w(n))}{\ln w(n)}\rightarrow p%
\text{,}
\]%
and for $p\neq -1$, we can apply Theorem \ref{thm5}.

In the case of $p=-1$, we have
\[
\alpha (n)+1=\frac{\ln (a_{n}/\bigtriangleup w(n))+\ln \ln n}{\ln \ln n}\sim
-\frac{1/2}{n\ln \ln n},
\]%
and%
\[
\ln w(n)(\alpha (n)+1)\sim -\frac{1}{2n}\rightarrow 0\text{.}
\]

Now Proposition \ref{prop7} $(ii)$ (with $C=1$) is applicable, and we arrive at $%
\sum_{i=1}^{n}a_{i}\sim \ln w(n)$.
\end{exm}

\bigskip

b) In the second type of results, we start from
\[
\frac{\ln (a_{n}/\bigtriangleup w(n))}{\ln w(n)}\rightarrow -1,
\]%
making the stronger assumption of existence of the following limit
\begin{equation}
\frac{\ln (a_{n}/\bigtriangleup w(n))+\ln w(n)}{\ln \ln w(n)}=\frac{\ln
(w(n)a_{n}/\bigtriangleup w(n))}{\ln \ln w(n)}\rightarrow \beta \text{.}
\label{7}
\end{equation}

\begin{prop}\label{prop8}
Assume that \eqref{7} holds.
\begin{enumerate}
\item [$(i)$] If $\beta <-1$, then $\sum_{i=1}^{\infty }a_{i}<\infty $ and%
\begin{equation}
\frac{\ln (\sum_{i=n}^{\infty }a_{i})}{\ln \ln w(n)}\rightarrow \beta +1\text{.%
}  \label{8}
\end{equation}

\item [$(ii)$] If $\beta >-1$, then $\sum_{i=1}^{\infty }a_{i}=\infty $ and%
\begin{equation}
\frac{\ln (\sum_{i=1}^{n}a_{i})}{\ln \ln w(n)}\rightarrow \beta +1\text{.}
\label{9}
\end{equation}
\end{enumerate}
\end{prop}

\begin{proof}
Assume that \eqref{7} holds. For $\epsilon >0$ we have
\[
\ln \left(\frac{w(n)a_{n}}{\bigtriangleup w(n)}\right)\leq (\beta +\epsilon )\ln \ln w(n), \quad n\geq
n%
{{}^\circ},%
\]%
and then%
\[
a_{n}\leq \bigtriangleup w(n)w^{\mathrm{i}}(n)(\ln w(n))^{\beta +\epsilon }\preceq
\int_{w(n)}^{w(n+1)}(\ln z)^{\beta +\epsilon }\frac{1}{z}\mathrm{d}z\text{.}
\]%
Similarly we have $$a_{n}\succeq \int_{w(n)}^{w(n+1)}(\ln z)^{\beta -\epsilon
}\frac{1}{z}\mathrm{d}z.$$

If $\beta <-1$, then $\sum_{i=1}^{\infty }a_{i}<\infty $, and $$(\ln
w(n))^{\beta -\epsilon +1}\preceq \sum_{i=n}^{\infty }a_{i}\preceq (\ln
w(n))^{\beta +\epsilon +1}.$$ Relation \eqref{8} follows.

If $\beta >-1$, then $\sum_{i=1}^{\infty }a_{i}=\infty $, and $$(\ln w(n))^{\beta
+1-\epsilon }\preceq \sum_{i=1}^{n}a_{i}\preceq (\ln w(n))^{\beta +1+\epsilon
}. $$ Relation \eqref{9} follows.
\end{proof}

\begin{exm}\label{ex8}

Consider

$$
a_n=\frac{(\ln\ln n)^p}{n\ln n}\textrm{,}
$$
and

$$
w(n)=\ln n\textrm{.}
$$

Note that $\Delta w(n)\sim1/n$, and we then obtain

$$
\frac{\ln(a_n/\Delta w(n))}{\ln\ln n}
=\frac{\ln((\ln\ln n)^p/\ln n)}{\ln\ln n}
=\frac{p\ln\ln\ln n-\ln\ln n}{\ln\ln n}
\to-1,
$$
as $n\to\infty$.
Also we have

$$
\frac{\ln(w(n)a_n/\Delta w(n))}{\ln\ln\ln n}
=\frac{\ln((\ln\ln n)^p)}{\ln\ln\ln n}
\to
p,
$$
as $n\to\infty$.
Hence, by applying Proposition \ref{prop8}, for $p<-1$ we obtain

$$
\sum_{i=1}^{\infty}a_i<\infty,
$$
and

$$
\frac{\ln \left(\sum_{i=n}^{\infty}a_i\right)}{\ln\ln\ln n}\to p+1, \ \text{as} \ n\to\infty.
$$
If $p>-1$, then

$$
\sum_{i=1}^{\infty}a_i=\infty,
$$
and

$$
\frac{\ln \left(\sum_{i=1}^{n}a_i\right)}{\ln\ln\ln n}\to p+1, \ \text{as} \ n\to\infty.
$$
If $p=-1$, we cannot arrive at the conclusion from Proposition \ref{prop8}.
\end{exm}

c) 
We prove a generalization
of an old result of Martin \cite{Martin}. Let $\ln _{(0)}z=z$, $\ln
_{(1)}z=\ln z$ and $\ln _{(k+1)}z=\ln \ln _{(k)}z$ for $k=1,2,...$ Note that for $%
k\geq 0$ we have
\[
(\ln _{(k+1)}(z))^{\prime }=\frac{(\ln _{(k)}(z))^{\prime }}{\ln _{(k)}(z)}=...=%
\frac{1}{z\times \ln _{(1)}z\times \ln _{(2)}z...\times \ln _{(k)}z}\text{.}
\]

If \eqref{7} holds with $\beta =-1$, Proposition \ref{prop8} cannot be used. In this case, we are to
replace \eqref{7} by the stronger assumption%
\[
\lim_{n\rightarrow \infty }\frac{\ln (w(n)a_{n}/\bigtriangleup w(n))-\ln
_{(2)}w(n)}{\ln _{(3)}w(n)}=\frac{\ln (w(n)\ln _{(1)}w(n)a_{n}/\bigtriangleup w(n))%
}{\ln _{(3)}w(n)}=\beta \text{.}
\]

As in Proposition \ref{prop8}, this leads to the case $\beta =-1$, at which we cannot make a decision. In
general, for $k=1,2,...$ we assume%
\[
\lim_{n\rightarrow \infty }\frac{\ln (w(n)\Pi _{i=1}^{k}\ln
_{(i)}w(n)a_{n}/\bigtriangleup w(n))}{\ln _{(k+2)}w(n)}=\beta _{k}\text{,}
\]
and consider the case of $\beta _{1}=\beta _{2}=...=\beta _{k-1}=-1$.

\begin{prop}\label{prop9}
Under the above assumptions we have as follows.
\begin{enumerate}
\item [$(i)$] If $\beta _{k}<-1$, then $\sum_{i=1}^{\infty }a_{i}<\infty $, and
\[
\lim_{n\to\infty}\frac{\ln (\sum_{i=n}^{\infty }a_{i})}{\ln _{(k+2)}w(n)}=\beta _{k}+1\text{.}
\]

\item [$(ii)$] If $\beta _{k}>-1$, then $\sum_{i=1}^{\infty }a_{i}=\infty $, and%
\[
\lim_{n\to\infty}\frac{\ln (\sum_{i=1}^{n}a_{i})}{\ln _{(k+2)}w(n)}=\beta _{k}+1\text{.}
\]

\item [$(iii)$] If $\beta _{k}=-1$, then assumptions $(i)$ and $(ii)$ should be taken for $k+1$, i.e. assumption $\beta _{k}<-1$ should be replaced by $\beta _{k+1}<-1$ and assumption $\beta _{k}>-1$ should be replaced by $\beta _{k+1}>-1$.

\end{enumerate}
\end{prop}

\begin{proof}
For $\epsilon >0$ we have
\[
\ln \left(\frac{w(n)\Pi _{i=1}^{k}\ln _{(i)}(w(n)a_{n})}{\bigtriangleup w(n)}\right)\leq (\beta
_{k}+\epsilon )\ln \ln _{(k+1)}w(n), \quad n\geq n%
{{}^\circ}%
\text{.}
\]%
It follows that%
\[
\frac{w(n)\Pi _{i=1}^{k}\ln _{(i)}(w(n)a_{n})}{\bigtriangleup w(n)}\leq (\ln
_{(k+1)}w(n))^{\beta _{k}+\epsilon }
\]%
and%
\[
a_{n}\leq \bigtriangleup w(n)\frac{(\ln _{(k+1)}w(n))^{\beta _{k}+\epsilon }}{%
w(n)\Pi _{i=1}^{k}\ln _{(i)}w(n)}\text{.}
\]%
It follows that
\[
a_{n}\preceq \int_{w(n)}^{w(n+1)}\frac{(\ln _{(k+1)}(z))^{\beta _{k}+\epsilon }%
}{z\Pi _{i=1}^{k}\ln _{(i)}z}\mathrm{d}z\text{.}
\]%
Similarly we find
\[
a_{n}\succeq \int_{w(n)}^{w(n+1)}\frac{(\ln _{(k+1)}(z))^{\beta _{k}-\epsilon }%
}{z\Pi _{i=1}^{k}\ln _{(i)}z}\mathrm{d}z\text{.}
\]

Now we consider the case $\beta _{k}<-1$. Using
\[
\int_{q}^{\infty }\frac{(\ln _{(k+1)}(z))^{\beta _{k}+\epsilon }}{z\Pi
_{i=1}^{k}\ln _{(i)}z}\mathrm{d}z=-\frac{(\ln _{(k+1)}(q))^{\beta _{k}+\epsilon +1}}{%
\beta _{k}+\epsilon +1}<\infty \text{,}
\]%
we find that $\sum_{i=1}^{\infty }a_{i}<\infty $, and%
\[
(\ln _{(k+1)}w(n))^{\beta _{k}-\epsilon +1}\preceq \sum_{i=n}^{\infty
}a_{i}\preceq (\ln _{(k+1)}w(n))^{\beta _{k}+\epsilon +1}\text{.}
\]%
Now it follows that
\[
\lim_{n\to\infty}\frac{\ln (\sum_{i=n}^{\infty }a_{i})}{\ln _{(k+2)}w(n)}=\beta _{k}+1\text{.}
\]%
In the case of $\beta _{k}>-1$, we find that $\sum_{i=1}^{\infty
}a_{i}=\infty $, and
\[
(\ln _{(k+1)}w(n))^{\beta _{k}-\epsilon +1}\preceq \sum_{i=1}^{n}a_{i}\preceq
(\ln _{(k+1)}w(n))^{\beta _{k}+\epsilon +1}\text{.}
\]%
Finally, we arrive at
\[
\lim_{n\to\infty}\frac{\ln (\sum_{i=1}^{n}a_{i})}{\ln _{(k+2)}w(n)}=\beta _{k}+1\text{.}
\]
\end{proof}

\begin{exm}
Assume that in Example \ref{ex8} we have $p=-1$.
According to Proposition \ref{prop8}, we cannot conclude on either convergence or divergence of $\sum_{i=1}^{n}a_i$ as $n\to\infty$.
From the above we have

$$
\frac{\ln((\ln n) a_n/\Delta w(n))}{\ln\ln\ln n}
\to-1,
$$
as $n\to\infty$.
Further, for $n>\exp(\exp(\exp(\exp(1))))$ we obtain

$$
\frac{\ln((\ln n)(\ln\ln n) a_n/\Delta w(n))}{\ln\ln\ln\ln n}
=\frac{\ln((\ln\ln n) (\ln\ln n)^{-1})}{\ln\ln\ln\ln n}=\frac{\ln 1}{\ln\ln\ln\ln n}
=0\textrm{.}
$$

Hence, Proposition \ref{prop9} allows us to conclude that

$$
\sum_{i=1}^{\infty}a_i=\infty,
$$
and

$$
\frac{\ln \left(\sum_{i=1}^{n}a_i\right)}{\ln\ln\ln\ln n}\to 0\textrm{,}
$$
as $n\to\infty$.
\end{exm}

\subsection{Results related to Theorem \ref{thm6}}

We define $\alpha (n)$%
\[
\alpha (n)=\frac{\bigtriangleup \ln (a_{n}/\bigtriangleup w(n))}{%
\bigtriangleup \ln w(n)}\text{.}
\]%
If $\alpha (n)\rightarrow \theta =-1$, Theorem \ref{thm6} does not provide information on convergence or divergence.
We provide three types of results.

a) Apparently, $(\alpha (n)+1)\bigtriangleup \ln w(n)=\bigtriangleup \ln
(a_{n}w(n)/\bigtriangleup w(n))$. Then taking the sums $\sum_{a}^{N}$ we obtain%
\[
\ln \left(\frac{w(N)}{\bigtriangleup w(N)}a_{N}\right)=c+\sum_{i=a}^{N}(\alpha (i)+1)\ln
w(i)
\]%
for some constant $c$, and%
\[
\frac{w(n)}{\bigtriangleup w(N)}a_{N}=C\exp \sum_{i=a}^{N}(\alpha (i)+1)\ln
w(i)
\]%
for some constant $C>0$.

\begin{prop}\label{prop10}
\begin{enumerate}
\item []
\item [$(i)$] If $\exp \sum_{i=a}^{N}(\alpha (i)+1)\ln (w(i))\geq B>0$, then $%
\sum_{i=a}^{n}a_{i}\succeq \ln w(n)$.

\item [$(ii)$] If $\exp \sum_{i=a}^{N}(\alpha (i)+1 )\ln (w(i))\rightarrow D$, then
$\sum_{i=1}^{\infty }a_{i}=\infty $, and $\sum_{i=1}^{n}a_{i}\sim E\ln w(n)$ for
some constant $E>0$.
\end{enumerate}
\end{prop}

\begin{proof}
$(i)$ If $\exp \sum_{i=a}^{N}(\alpha (i)+1 )\ln w(i)\geq B$, then
\[
a_{N}\succeq \frac{\bigtriangleup w(N)}{w(N)}\succeq
\int_{w(N)}^{w(N+1)}\frac{1}{z}\mathrm{d}z\text{.}
\]%
It follows that $\sum_{i=a}^{n}a_{i}\succeq \int_{w(a)}^{w(n+1)}z^{-1}\mathrm{d}z\asymp
\ln w(n)$, and hence we find that \\
$\sum_{i=a}^{n}a_{i}\succeq \ln w(n)$.

$(ii)$ If $\exp \sum_{i=a}^{N}(\alpha (i)+1 )\ln w(i)\rightarrow D$
(finite), then
\[
\frac{w(N)}{\bigtriangleup w(N)}a_{N}\rightarrow CD:=E\text{.}
\]%
It follows that $a_{N}\sim E\int_{w(N)}^{w(N+1)}z^{-1}\mathrm{d}z$. The result follows by summation.
\end{proof}

\bigskip

b) To obtain the second type of results similar to that is given in \eqref{7}, we assume%
\begin{equation}
\lim_{n\rightarrow \infty }\frac{\bigtriangleup \ln
(w(n)a_{n}/\bigtriangleup w(n))}{\bigtriangleup \ln \ln w(n)}=\beta \text{.}
\label{10}
\end{equation}%

\begin{prop}\label{prop11}
Assume that \eqref{10} holds.
\begin{enumerate}
\item [$(i)$] If $\beta <-1$, then $\sum_{i=1}^{\infty }a_{i}<\infty $ and
\[
\lim_{n\rightarrow \infty }\frac{\ln (\sum_{i=n}^{\infty }a_{i})}{\ln \ln w(n)}%
=\beta +1\text{.}
\]

\item[$(ii)$] If $\beta >-1$, then $\sum_{i=1}^{\infty }a_{i}=\infty $, and%
\[
\lim_{n\rightarrow \infty }\frac{\ln (\sum_{i=1}^{n}a_{i})}{\ln \ln w(n)}%
=\beta +1\text{.}
\]
\end{enumerate}
\end{prop}

\begin{proof}
From \eqref{10} it follows that for $\epsilon >0$ we have%
\[
\beta -\epsilon \leq \frac{\bigtriangleup \ln (w(n)a_{n}/\bigtriangleup w(n))%
}{\bigtriangleup \ln \ln w(n)}\leq \beta +\epsilon, \quad n\geq n%
{{}^\circ}%
\text{.}
\]%
Hence,
\[
(\beta -\epsilon )\bigtriangleup \ln \ln w(n)\leq \bigtriangleup \ln \frac{%
w(n)a_{n}}{\bigtriangleup w(n)}\leq (\beta +\epsilon )\bigtriangleup \ln \ln
w(n), \quad n\geq n%
{{}^\circ}%
\text{.}
\]%
Taking the sums $\sum_{n%
{{}^\circ}%
}^{N}$ leads to
\[
C+(\beta -\epsilon )\ln \ln w(N)\leq \ln \frac{w(N)a_{N}}{\bigtriangleup w(N)%
}\leq D+(\beta +\epsilon )\ln \ln w(N), \quad N\geq n%
{{}^\circ}%
\text{,}
\]%
and
\[
\frac{\bigtriangleup w(N)}{w(N)}(\ln w(N))^{\beta -\epsilon }\preceq
a_{N}\preceq \frac{\bigtriangleup w(N)}{w(N)}(\ln w(N))^{\beta +\epsilon }%
\text{.}
\]%
It follows that
\[
\int_{w(N)}^{w(N+1}\frac{1}{z}(\ln z)^{\beta -\epsilon }\mathrm{d}z\preceq
a_{N}\preceq \int_{w(N)}^{w(N+1)}\frac{1}{z}(\ln z)^{\beta +\epsilon }\mathrm{d}z%
\text{.}
\]
The results follows.
\end{proof}

\bigskip

c) As in the previous subsection, we can obtain a hierarchy of results. Note
that \eqref{10} reduces to

\[
\lim_{n\rightarrow \infty }\frac{\bigtriangleup \ln
(w(n)a_{n}/\bigtriangleup w(n))}{\bigtriangleup \ln _{(2)}w(n)}=\beta \text{.}
\]%
Now we make the following assumption: for $k\geq 1$ assume that%
\[
\lim_{n\rightarrow \infty }\frac{\bigtriangleup \ln (w(n)\Pi _{i=1}^{k}\ln
_{(i)}w(n)a_{n}/\bigtriangleup w(n))}{\bigtriangleup \ln _{(k+2)}w(n)}=\beta _{k}%
\text{,}
\]%
where $\beta _{1}=...=\beta _{k-1}=-1$.

\begin{prop}\label{prop12}
Under the above assumptions we have
\begin{enumerate}
\item [$(i)$] If $\beta _{k}<-1$, then $\sum_{i=1}^{\infty }a_{i}<\infty $, and
\[
\lim_{n\to\infty}\frac{\ln (\sum_{i=n}^{\infty }a_{i})}{\ln _{(k+2)}w(n)}=\beta _{k}+1\text{.}
\]

\item [$(ii)$] If $\beta _{k}>-1$, then $\sum_{i=1}^{\infty }a_{i}=\infty $, and%
\[
\lim_{n\to\infty}\frac{\ln (\sum_{i=1}^{n}a_{i})}{\ln _{(k+2)}w(n)}=\beta _{k}+1\text{.}
\]

\item [$(iii)$] If $\beta _{k}=-1$, then assumptions $(i)$ and $(ii)$ should be taken for $k+1$, i.e. assumption $\beta _{k}<-1$ should be replaced by $\beta _{k+1}<-1$ and assumption $\beta _{k}>-1$ should be replaced by $\beta _{k+1}>-1$.
\end{enumerate}
\end{prop}

\begin{proof}
For $\epsilon >0$ we have%
\[
\bigtriangleup \ln \left(\frac{w(n)\Pi _{i=1}^{k}\ln _{(i)}w(n)a_{n}}{\bigtriangleup
w(n)}\right)\leq (\beta _{k}+\epsilon )\bigtriangleup \ln _{(k+2)}w(n), \quad n\geq n%
{{}^\circ}%
\text{.}
\]%
Taking the sums $\sum_{n%
{{}^\circ}%
}^{N}$ yields%
\[
\ln \left(\frac{w(N)\Pi _{i=1}^{k}\ln_{(i)}w(N)a_{N}}{\bigtriangleup w(N)}\right)\leq A+(\beta
_{k}+\epsilon )\ln _{(k+2)}w(N), \quad N>n%
{{}^\circ}%
\text{,}
\]%
and%
\[
a_{N}\preceq \bigtriangleup w(N)\frac{(\ln _{(k+1)}w(N))^{\beta _{k}+\epsilon }%
}{w(N)\Pi _{i=1}^{k}\ln _{(i)}w(N)}\preceq \int_{w(N)}^{w(N+1}\frac{(\ln
_{(k+1)}(z))^{\beta _{k}+\epsilon }}{z\Pi _{i=1}^{k}\ln _{(i)}z}\mathrm{d}z\text{.}
\]%
Similarly,%
\[
a_{N}\succeq \int_{w(N)}^{w(N+1}\frac{(\ln _{(k+1)}(z))^{\beta _{k}-\epsilon }%
}{z\Pi _{i=1}^{k}\ln _{(i)}z}\mathrm{d}z\text{.}
\]
Now, the result follows by summation.
\end{proof}

\section{Concluding remarks}\label{S5}

1) In this paper we studied the consequences of the assumptions
\[
\lim_{n\rightarrow \infty }\frac{\ln a_{n}}{\ln w(n)}=\theta,
\]%
and%
\[
\lim_{n\rightarrow \infty }\frac{\bigtriangleup \ln a_{n}}{\bigtriangleup
\ln w(n)}=\theta \text{.}
\]%
It could be interesting to study assumptions of the type%
\[
\lim_{n\rightarrow \infty }\frac{\bigtriangleup ^{2}\ln a_{n}}{%
\bigtriangleup ^{2}\ln w(n)}=\theta \text{,}
\]%
or higher order differences.

2) When studying functions, we can also consider statements of the form%
\[
\lim_{x\rightarrow \infty }\frac{\ln f(x)}{\ln w(x)}=\theta \text{,}
\]%
or%
\[
\lim_{x\rightarrow \infty }\frac{(\ln f(x))^{\prime }}{(\ln w(x))^{\prime }}%
=\theta \text{.}
\]%
In the first case we proved that there exist functions $A(x),B(x)\in
RV_{\theta }$ so that $A(w(x))\preceq f(x)\preceq B(w(x))$. In the second
case, we found that $f(x)=h(w(x))$, where $h(x)\in RV_{\theta }$.

3) Along with the cases where the limits exist, we considered the cases where the limits are replaced with $\limsup $ and $\liminf $. For example we have the following statements.

\begin{prop}\label{prop13}
\begin{enumerate}
\item []
\item [$(i)$] Assume that
\[
\limsup_{n\rightarrow \infty }\frac{\ln (a_{n}/\bigtriangleup w(n))}{\ln
w(n)}=\theta <-1\text{,}
\]%
then $\sum_{i=1}^{\infty }a_{i}<\infty $.

\item [$(ii)$] Assume that
\[
\liminf_{n\rightarrow \theta }\frac{\ln (a_{n}/\bigtriangleup w(n))}{\ln
w(n)}=\theta >-1\text{,}
\]%
then $\sum_{i=1}^{\infty }a_{i}=\infty $.
\end{enumerate}
\end{prop}

\begin{prop}\label{prop14}
Let $f(x)=a_{\left[ x\right] }$, and assume that
\[
\alpha =\liminf_{n\rightarrow \infty }\frac{\bigtriangleup \ln a_{n}}{%
\bigtriangleup \ln w(n)}\leq \limsup_{n\rightarrow \infty }\frac{%
\bigtriangleup \ln a_{n}}{\bigtriangleup \ln w(n)}=\beta \text{.}
\]%
Then, for all $t\geq 1$, we have $$\liminf_{x\rightarrow \infty
}\frac{f(tx)}{f(x)}\geq t^{\alpha },$$ and $$\limsup_{x\rightarrow \infty
}\frac{f(tx)}{f(x)}\leq t^{\beta }.$$
\end{prop}

Proposition \ref{prop14} shows that $f(x)$ is in the class of so-called $O-$%
regularly varying functions studied among the others in \cite{Bingham}.


%
 \section*{Conflict of interest}
%
No potential conflict of interests was reported by the authors.

\section{List of references}

%

\bibliographystyle{amsplain}

\begin{thebibliography}{n} 

\bibitem{Abramov}
{V.\,M. Abramov}, \emph{Extension of the Bertrand-De Morgan test
and its application}, {Amer. Math. Monthly} \textbf{127} (2020), 
444--448. 

\bibitem{Bingham}
{N. Bingham, C. Goldie, J. Teugels}, \emph{Regular Variation}, Cambridge
University Press, 1989.

\bibitem{BS}
{R. Bojanic, E. Seneta}, \emph{A unified theory of regularly varying
sequences}, {Math. Zeitschrift} \textbf{134} (1973), 91--106. 

\bibitem{BBNV}
{L. Bourchtein, A. Bourchtein, G. Nornberg, C. Venzke}, \emph{A
hierarchy of the convergence tests related to Cauchy's test}, {Int. J.
Math. Analysis} \textbf{6} (2012), 
1847--1869.

\bibitem{CKO}
{M. Cadena, M. Kratz, E. Omey}, \emph{On the order of functions at
infinity},  {J. Math. Anal. Appl.} \textbf{452} (2017), 
109--125. 

\bibitem{CKO1}
{M. Cadena, M. Kratz, E. Omey}, \emph{Characterization of a general class of tail probability
distributions}, {Statist. Probab. Lett.} \textbf{154} (2019). 

\bibitem{Cauchy}
{A.\,L. Cauchy}, \emph{Cours d'Analyse de l'Ecole Royale
Polytechnique}, L'Imprimerie Poyale, Debure Fr\`{e}res, Paris, 1821.

\bibitem{Daduna}
{H. Daduna}, \emph{Alternating birth-and-death processes}, Preprint, 2020; a version available at \url{https://arxiv.org/pdf/2004.08816.pdf}.

\bibitem{d'A}
{J. d'Alembert}, \emph{Reflexions fur les Suites divergentes ou convergentes}, in: J. d'Alembert (ed.), \emph{Opuscules Math\'ematiques V5:
ou M\'emoires sur Diff\'erens Sujets de G\'eom\'etrie, de M\'echanique, etc.}, 1768, 171--183.

\bibitem{Hammond}
{C.\,N.\,B. Hammond}, \emph{The case of Raabe's test}, {Math. Mag.} \textbf{93} (2020),
36--46. 

\bibitem{Martin}
{M. Martin}, \emph{A sequence of limit tests for the convergence of series},
{Bull. Amer. Math. Soc.} \textbf{47} (1941), 
452--457.

\bibitem{Rao}
{K.\,P.\,S.\,B. Rao},\emph{A new (?) test for convergence of series},
{Math. Mag.} \textbf{67} (1994), 
301--302. 

\bibitem{Rehak1}
{P. \v{R}eh\'ak}, \emph{Refined discrete regular variation and its applications},
{Math. Meth. Appl. Sci.} \textbf{42} (2019), 6009--6020. 

\bibitem{Rehak2}
{P. \v{R}eh\'ak}, \emph{Kummer test and regular variation}, {Monatsh.
Math.} \textbf{192} (2020), 419--426. 

%
%
%
%
%
%
%

\end{thebibliography}

\end{document}